\documentclass[11pt]{amsart}

\usepackage[T1]{fontenc}
\usepackage{lmodern}
\usepackage{amsmath,amssymb,mathtools}
\usepackage{enumitem}
\usepackage{microtype}
\usepackage[colorlinks=true,linkcolor=blue,citecolor=blue,urlcolor=blue]{hyperref}
\hypersetup{
  pdftitle={A Compactness Characterization of Strongly Symmetric Homeomorphisms},
  pdfauthor={Tailiang Liu and Yuliang Shen},
  pdfsubject={A compactness characterization of strong symmetry and asymptotically smooth welding},
  pdfkeywords={BMOA, strongly quasisymmetric homeomorphism, strongly symmetric homeomorphism, asymptotically smooth curve}
}

\newtheorem{theorem}{Theorem}[section]
\newtheorem{proposition}[theorem]{Proposition}
\newtheorem{lemma}[theorem]{Lemma}
\newtheorem{corollary}[theorem]{Corollary}
\theoremstyle{definition}

\newcommand{\D}{\mathbb D}
\newcommand{\Hh}{\mathbb H}
\newcommand{\T}{\mathbb T}
\newcommand{\R}{\mathbb R}

\newcommand{\BMO}{\mathrm{BMO}}
\newcommand{\VMO}{\mathrm{VMO}}
\newcommand{\CMO}{\mathrm{CMO}}
\newcommand{\BMOA}{\mathrm{BMOA}}
\newcommand{\VMOA}{\mathrm{VMOA}}
\newcommand{\SQS}{\mathrm{SQS}}
\newcommand{\SSS}{\mathrm{SS}}
\newcommand{\essinf}{\operatorname*{ess\,inf}}
\newcommand{\esssup}{\operatorname*{ess\,sup}}
\newcommand{\supp}{\operatorname{supp}}
\newcommand{\id}{\mathrm{Id}}
\newcommand{\Kop}{\mathcal K}
\newcommand{\dist}{\operatorname{dist}}

\title[A Compactness Characterization of Strongly Symmetric Homeomorphisms]
{A Compactness Characterization of Strongly Symmetric Homeomorphisms}

\author{Tailiang Liu}
\address{School of Mathematics and Physics, Jiangsu University of Technology, Changzhou 213001, P. R. China}
\email{ltlmath@163.com}

\author{Yuliang Shen}
\address{Department of Mathematics, Soochow University, Suzhou 215006, P. R. China}
\email{ylshen@suda.edu.cn}
\date{August 2026}


\begin{document}

\begin{abstract}
A self-homeomorphism $h$ of the unit circle $\T$ is strongly symmetric if it is absolutely continuous and $\log h'\in\VMO(\T)$. Let $P_h^-$ be the anti-analytic component of the pullback operator $ P_h: F\mapsto= F\circ h$ on $\BMOA(\D)$,  where $\D$ is the unit disk. P. Jones proved $P_h$ is bounded  on BMO if and only if $h$ is strongly quasisymmetric.  Fan, Hu, and Shen showed that the strong symmetry of $h$ yields the compactness of $P_h^-$, and raised the question of whether the converse is true. We answer this question affirmatively, establishing that $P_h^-$ is compact if and only if $\log h'\in\VMO(\T)$. This  provides a compactness characterization  of the VMO-Teichm\"uller space.
\end{abstract}

\maketitle

\section{Introduction and statement of the problem}\label{sec:introduction}
If $I\subset\T$ is an arc and $u\in L^1(\T)$, write
\[
 u_I=\frac1{|I|}\int_Iu\,|d\zeta|,
 \qquad
 M_{\T}(u,I)=\frac1{|I|}\int_I|u-u_I|\,|d\zeta|.
\]
The space $\BMO(\T)$ consists of those $u$ for which
\[
 \|u\|_{\BMO(\T)}=\sup_I M_{\T}(u,I)<\infty,
\]
and $u\in\VMO(\T)$ when
\[
 \lim_{r\downarrow0}\sup_{|I|<r}M_{\T}(u,I)=0.
\]  Equivalently, $\VMO(\T)$ is the closure of $C(\T)$ in the $\BMO(\T)$ norm.

Throughout the paper, unless otherwise stated, every BMO space is understood modulo constants and is equipped with the BMO seminorm.

The geometric background of the problem comes from the theory of asymptotically smooth curves.  A rectifiable Jordan curve $\Gamma$ is asymptotically smooth if, for the shorter subarc $\widetilde{z_1z_2}\subset\Gamma$ joining $z_1$ and $z_2$,
\[
 \lim_{t\downarrow0}
 \sup_{\substack{z_1,z_2\in\Gamma\\ |z_1-z_2|\le t}}
 \frac{\operatorname{length}(\widetilde{z_1z_2})}{|z_1-z_2|}=1.
\]
This is the vanishing version of the chord-arc condition.  Sarason's introduction of $\VMO$ provided the natural function-space scale for such a boundary condition \cite{Sarason}.  Pommerenke then proved that, for a conformal map $f$ of the disk onto a Jordan domain, asymptotic smoothness of the boundary is equivalent to the analytic condition
\[
 \log f'\in\VMOA(\D)
\]
(with rectifiability contained in the geometric condition) \cite{Pommerenke}.

The same phenomenon was subsequently expressed through quasiconformal extensions and vanishing Carleson measures.  Dyn'kin obtained estimates of $f''/f'$ in terms of the complex dilatation of an extension, providing a systematic bridge between boundary smoothness and Carleson-type decay \cite{Dynkin}.  In the universal Teichm\"uller setting, Shen and Wei assembled these analytic, geometric, and welding descriptions: strong symmetry of the welding homeomorphism, vanishing Carleson control of a Beltrami coefficient, asymptotic smoothness of the welding curve, and the $\VMOA$ condition for the logarithmic derivative are equivalent formulations of the same vanishing theory; see \cite{ShenWei}.  In particular, a Jordan curve is asymptotically smooth if and only if its conformal welding  is strongly symmetric. A conformally invariant real-line version can be found in \cite{LiuShen}.  

Let $h:\T\to\T$ be an orientation-preserving homeomorphism.  The homeomorphism $h$ is called strongly quasisymmetric, and one writes $h\in\SQS(\T)$, when $h$ is absolutely continuous and $h'$ is an $A_\infty$ weight.  Equivalently, $\log h'\in\BMO(\T)$ and the pull-back operator
\[
 C_hu=u\circ h
\]
is a bounded isomorphism of $\BMO(\T)$; see Jones's theorem \cite{Jones}.  The homeomorphism $h$ is called strongly symmetric, and one writes $h\in\SSS(\T)$, when $h$ is absolutely continuous and
\[
 \log h'\in\VMO(\T).
\]
Such a homeomorphism is automatically strongly quasisymmetric by John-Nirenberg; see also \cite[p.5]{LiuShen}.  
Recall that  $\BMOA(\D)$ consists of the analytic functions that are Poisson extensions of BMO boundary functions.
  For $F\in\BMOA(\D)$, the analytic and anti-analytic parts of $P_{\D}(F\circ h)$ determine two functions $P_h^+F,P_h^-F\in\BMOA(\D)$ by
\begin{equation}
 P_{\D}(F\circ h)
 =P_h^+F+\overline{P_h^-F}.
 \label{eq:def-Ph-circle-intro}
\end{equation}
Thus $P_h^+F$ is the analytic component, whereas $P_h^-F$ is the analytic function obtained by conjugating the anti-analytic component.  The corresponding Hilbert-transform formulas are given in Section~\ref{sec:projections}.  The operator $P_h^-$ is the one studied by Fan, Hu and Shen.  For recent work on the analytic component $P_h^+$ and its relation to chord-arc conformal welding, see the first author's paper \cite{LiuChordArc}.  Fan, Hu and Shen proved the forward implication in the following statement and posed the converse as Conjecture~4.1 in \cite{FHS}.

\begin{theorem}\label{thm:main}
Let $h\in\SQS(\T)$.  Then $h\in\SSS(\T)$ if and only if
\[
 P_h^-:\BMOA(\D)\longrightarrow\BMOA(\D)
\]
is compact.
\end{theorem}

For convenience, we will carry out our proof on the real line, as the case of the circle can be deduced via a Möbius transformation. Accordingly, Section 2 details this reduction, and Section 3 presents the proof of the theorem. While this problem arises from the VMO theory of the universal Teichmüller space, our approach combines the geometric control of quasisymmetry with localized testing and level-set separation from classical Calderón-Zygmund theory.

\section{Conformal invariance of the projection operator}\label{sec:projections}

Recall that for a locally integrable function $q \in L^1_{\mathrm{loc}}(\R)$ and an interval $I\subset\R$, we also set$$q_I=\frac1{\vert{}I\vert{}}\int_Iq(x)\,dx,  \qquad  M(q,I)=\frac1{\vert{}I\vert{}}\int_I\vert{}q-q_I\vert{}\,dx.$$The space $\BMO(\R)$ is then defined as the set of functions $q$ for which $\sup_I M(q,I) < \infty$.
The space $\CMO(\R)$ is the BMO closure of $C_c(\R)$.  By the standard characterization recalled in \cite{Uchiyama}, $q\in\CMO(\R)$ if and only if
\begin{align}
	\lim_{r\downarrow0}\sup_{|I|<r}M(q,I)&=0, \label{eq:CMO-small}\\
	\lim_{R\to\infty}\sup_{|I|>R}M(q,I)&=0, \label{eq:CMO-large}\\
	\lim_{|x|\to\infty}M(q,I_0+x)&=0
	\quad\text{for every fixed interval }I_0. \label{eq:CMO-tail}
\end{align}
The conformal correspondence
\[
f\in\VMO(\T)
\quad\Longleftrightarrow\quad
f\circ\gamma\in\CMO(\R),
\]
where $\gamma(z)=(z-i)/(z+i)$ maps $\Hh$ onto $\D$, was also established in \cite{LiuShen}.

Let $H_{\mathbb{T}}$ be the circular Hilbert transform, defined by the principal value integral$$H_{\mathbb{T}}u(\theta)=\frac{1}{2\pi}\mathrm{p.v.}\int_{0}^{2\pi}u(t)\cot\left(\frac{\theta-t}{2}\right)dt.$$
The operators
\begin{equation}
 \frac12(\id+iH_{\T})
 \quad\text{and}\quad
 \frac12(\id-iH_{\T})
 \label{eq:circle-projections}
\end{equation}
are the bounded projections on $\BMO(\T)$ onto the analytic and anti-analytic boundary parts, respectively.  Their ranges are the boundary values of $\BMOA(\D)$ and of its complex conjugate.

For $F\in\BMOA(\D)$, define $P_h^+F,P_h^-F\in\BMOA(\D)$ by the boundary identities
\begin{equation}
 P_h^+F=\frac12(\id+iH_{\T})(F\circ h),
 \qquad
 \overline{P_h^-F}=\frac12(\id-iH_{\T})(F\circ h).
 \label{eq:Ph-H-circle-boundary}
\end{equation}
After Poisson extension, these identities are exactly
\[
 P_{\D}(F\circ h)=P_h^+F+\overline{P_h^-F},
\]
which is \eqref{eq:def-Ph-circle-intro}.

The same convention applies on the real line.  Let
\[
 H_{\R}f(x)=\frac1\pi\operatorname{p.v.}\int_{\R}f(y)\Big(\frac{1}{x-y}+\frac{y}{1+y^2} \Big)\,dy
\]
be the Hilbert transform on $\R$.  If $k\in\SQS(\R)$ and $F\in\BMOA(\Hh)$, define $P_{k,\R}^+F,P_{k,\R}^-F\in\BMOA(\Hh)$ by
\begin{equation}
 P_{k,\R}^+F=\frac12(\id+iH_{\R})(F\circ k),
 \qquad
 \overline{P_{k,\R}^-F}=\frac12(\id-iH_{\R})(F\circ k).
 \label{eq:Ph-H-line-boundary}
\end{equation}
Equivalently,
\[
 P_{\Hh}(F\circ k)
 =P_{k,\R}^+F+\overline{P_{k,\R}^-F}.
\]
Proposition~\ref{prop:circle-line-transfer} below shows that these real-line operators are precisely the Cayley transforms of the circle operators.

By composing $h$ on the left with a rotation, which changes neither $\log h'$ nor the compactness of $P_h^-$, we may and do assume that h(1)=1.
Set
\begin{equation}
 k=\gamma^{-1}\circ h\circ\gamma.
 \label{eq:def-k}
\end{equation}
Then $k$ is an increasing self-homeomorphism of $\R$ fixing infinity. 
Define \[\SSS_0(\R)=\{k\in\SQS(\R):\log k'\in\CMO(\R).  \}\]
 The conformal invariance of the BMO--Teichm\"uller space gives $k\in\SQS(\R)$ (see \cite[Theorem~1.2]{LiuShen}); moreover,
\begin{equation}
 k\in\SSS_0(\R)
 \quad\Longleftrightarrow\quad
 h\in\SSS(\T);
 \label{eq:SS-circle-line}
\end{equation}
 In particular, $\log k'\in\BMO(\R)$.

Set
\[
 U_{\gamma}F=F\circ\gamma.
\]
Composition with $\gamma$ is a bounded isomorphism from $\BMO(\T)$ onto $\BMO(\R)$ and from $\BMOA(\D)$ onto $\BMOA(\Hh)$.  The identity
\[
 h\circ\gamma=\gamma\circ k
\]
and conformal invariance of Poisson extension give
\begin{align*}
 P_{\Hh}\bigl((F\circ\gamma)\circ k\bigr)
 &=P_{\Hh}\bigl((F\circ h)\circ\gamma\bigr)\\
 &=\bigl(P_{\D}(F\circ h)\bigr)\circ\gamma.
\end{align*}
Comparing the analytic and anti-analytic parts on the two sides yields
\begin{equation}
 P_{k,\R}^{\pm}(F\circ\gamma)
 =(P_h^{\pm}F)\circ\gamma.
 \label{eq:Ph-circle-line}
\end{equation}
By \eqref{eq:Ph-H-circle-boundary},
\begin{equation}
 H_{\R}U_{\gamma}=U_{\gamma}H_{\T}
 \quad\text{on }\BMO(\T).
 \label{eq:Hilbert-intertwine}
\end{equation}
Thus both the $P_h^\pm$ operators and the Hilbert transforms are carried from the circle to the line by the same Cayley transform.

\begin{proposition}[Circle-line compactness transfer]\label{prop:circle-line-transfer}
Let $h\in\SQS(\T)$.  Then $P_h^-$ is compact on $\BMOA(\D)$ if and only if $P_{k,\R}^-$ is compact on $\BMOA(\Hh)$.
\end{proposition}

\begin{proof}
Equation \eqref{eq:Ph-circle-line} gives
\[
 P_{k,\R}^-=U_{\gamma}P_h^-U_{\gamma}^{-1}.
\]
Since $U_{\gamma}$ is a bounded isomorphism, compactness is preserved in both directions. 
\end{proof}

The following computation applies on either $\T$ or $\R$.  Write $H_X$ for the corresponding Hilbert transform, where $X\in\{\T,\R\}$, and write $P_g^-$ for the operator defined above; on the real line this means $P_{g,\R}^-$.  Every element of $\BMO(X)$ can be written, modulo constants, as
\[
 u=F+\overline G,
\]
where $F$ and $G$ are  boundary values of BMOA  in the corresponding disk or half-plane.

\begin{proposition}\label{prop:Ph-H-relation}
Let $g\in\SQS(X)$.  For boundary values of  BMOA functions $F$ and $G$,
\begin{align}
 [H_X,C_g]F&=2i\,\overline{P_g^-F},
 \label{eq:commutator-on-analytic}\\
 [H_X,C_g]\overline G&=-2i\,P_g^-G.
 \label{eq:commutator-on-antianalytic}
\end{align}
Consequently, the following are equivalent:
\begin{enumerate}[label=\textup{(\roman*)}]
\item $P_g^-$ is compact on BMOA;
\item $[H_X,C_g]$ is compact on $\BMO(X)$;
\item
\[
 \Kop_{g,X}=C_gH_XC_{g^{-1}}-H_X
\]
is compact on $\BMO(X)$.
\end{enumerate}
\end{proposition}

\begin{proof}
Since $H_XF=-iF$, formula \eqref{eq:Ph-H-circle-boundary} or \eqref{eq:Ph-H-line-boundary} gives
\begin{align*}
 [H_X,C_g]F
 &=H_X(F\circ g)+i(F\circ g)\\
 &=2i\,\frac12(\id-iH_X)(F\circ g)
 =2i\,\overline{P_g^-F}.
\end{align*}
Similarly, $H_X\overline G=i\overline G$.  Since composition by $g$ commutes with complex conjugation,
\begin{align*}
 [H_X,C_g]\overline G
 &=H_X(\overline{G\circ g})-i\overline{G\circ g}\\
 &=-2i\,\frac12(\id+iH_X)(\overline{G\circ g})
 =-2i\,P_g^-G.
\end{align*}
Therefore, for $u=F+\overline G$,
\begin{equation}
 [H_X,C_g]u
 =2i\,\overline{P_g^-F}-2i\,P_g^-G.
 \label{eq:commutator-full-decomposition}
\end{equation}
The analytic and anti-analytic projections are bounded on BMO, so \eqref{eq:commutator-full-decomposition} proves the equivalence of \textup{(i)} and \textup{(ii)}.  Finally,
\begin{equation}
 \Kop_{g,X}=-[H_X,C_g]C_{g^{-1}}.
 \label{eq:K-commutator}
\end{equation}
Since $C_{g^{-1}}$ is a bounded isomorphism of BMO, \textup{(ii)} and \textup{(iii)} are equivalent.
\end{proof}

\section{Proof of Theorem~\ref{thm:main}}\label{sec:kernel}

\begin{lemma}[Kernel formula]\label{lem:kernel}
Let $k\in\SQS(\R)$, let $f\in L^\infty(\R)$ have compact support, and let $x\notin\supp f$.  Then
\begin{equation}
 \Kop_kf(x)
 =\frac1\pi\int_{\R}f(y)
 \left\{
 \frac{k'(y)}{k(x)-k(y)}-\frac1{x-y}
 \right\}\,dy.
 \label{eq:kernel}
\end{equation}
The integral is absolutely convergent.
\end{lemma}
This proof is simple. Since $x \notin \text{supp } f$, no principal value is needed. The formula follows immediately from the definition of $\mathcal{K}_k$ and the change of variables $s = k(y)$.

A standard elementary fact is recorded first.

\begin{lemma}\label{lem:average-functional}
Let $A,B\subset\R$ be fixed bounded intervals.  There is a constant $C(A,B)$ such that
\[
 |u_A-u_B|\le C(A,B)\|u\|_{\BMO(\R)}
\]
for every $u\in\BMO(\R)$.  Consequently, $[u]\mapsto u_A-u_B$ is a continuous linear functional on $\BMO(\R)$.
\end{lemma}

\begin{proof}
Choose a bounded interval $Q$ containing $A\cup B$.  Since the intervals are fixed,
\[
 |u_A-u_Q|
 \le\frac1{|A|}\int_A|u-u_Q|
 \le\frac{|Q|}{|A|}\|u\|_{\BMO},
\]
and the same estimate holds for $B$. Then $$\vert{}u_A - u_B\vert{} \le \vert{}u_A - u_Q\vert{} + \vert{}u_Q - u_B\vert{}\le C(A,B)\|u\|_{\BMO(\R)}.$$
\end{proof}

\subsection{Supports shrinking to one point}

\begin{lemma}[Shrinking-support localization]\label{lem:shrinking}
Assume that $\Kop_k$ is compact on $\BMO(\R)$.  Let $I_n$ be bounded intervals such that
\[
 |I_n|\longrightarrow0,
 \qquad
 c(I_n)\longrightarrow x_0\in\R,
\]
and let $f_n$ satisfy
\[
 \|f_n\|_\infty\le1,
 \qquad
 \supp f_n\subset I_n.
\]
Then
\[
 \|\Kop_kf_n\|_{\BMO}\longrightarrow0.
\]
\end{lemma}

\begin{proof}
The sequence $(f_n)$ is bounded in BMO, since $\|f_n\|_{\BMO}\le2$.  Let $L\Subset\R\setminus\{x_0\}$ be a fixed compact interval.  For all sufficiently large $n$, $L\cap I_n=\varnothing$.  By Lemma~\ref{lem:kernel},
\[
 \sup_{x\in L}\left|\int_{I_n}\frac{f_n(y)}{x-y}\,dy\right|
 \le\frac{|I_n|}{\dist(L,I_n)}\longrightarrow0.
\]
Similarly,
\[
 \sup_{x\in L}\left|\int_{I_n}
 \frac{f_n(y)k'(y)}{k(x)-k(y)}\,dy\right|
 \le
 \frac{|k(I_n)|}{\dist(k(L),k(I_n))}\longrightarrow0,
\]
because $k$ is continuous and $|k(I_n)|\to0$.  Hence
\begin{equation}
 \Kop_kf_n\longrightarrow0
 \quad\text{uniformly on compact subsets of }\R\setminus\{x_0\}. \label{eq:local-uniform-zero}
\end{equation}

Suppose the conclusion fails.  By compactness, after passing to a subsequence,
\[
 [\Kop_kf_n]\longrightarrow[u]
 \quad\text{in }\BMO.
\]
If $A$ and $B$ are bounded intervals avoiding $x_0$, then \eqref{eq:local-uniform-zero} and Lemma~\ref{lem:average-functional} give
\[
 u_A-u_B
 =\lim_{n\to\infty}
 \big((\Kop_kf_n)_A-(\Kop_kf_n)_B\big)=0.
\]
By Lebesgue differentiation, $u$ is a.e. constant on $\R\setminus\{x_0\}$, and hence represents the zero element of $\BMO$.  Thus every norm-convergent subsequence of $[\Kop_kf_n]$ converges to zero, contradicting the assumed failure of norm convergence to zero.
\end{proof}

\subsection{Supports escaping to infinity}

For $R>0$ define
\begin{equation}
 \beta_k(R)=
 \sup\left\{
 \|\Kop_kf\|_{\BMO}:
 \|f\|_\infty\le1,
 \ \supp f\Subset\R\setminus[-R,R]
 \right\}. \label{eq:def-beta}
\end{equation}

\begin{lemma}[Tail localization]\label{lem:tail}
If $\Kop_k$ is compact on $\BMO(\R)$, then
\[
 \beta_k(R)\longrightarrow0
 \qquad(R\to\infty).
\]
\end{lemma}

\begin{proof}
Suppose otherwise.  Then there are $R_n\to\infty$, functions $f_n$ with
\[
 \|f_n\|_\infty\le1,
 \qquad
 \supp f_n\Subset\R\setminus[-R_n,R_n],
\]
and a number $\varepsilon>0$ such that
\[
 \|\Kop_kf_n\|_{\BMO}\ge\varepsilon.
\]
Fix a compact interval $L\subset[-A,A]$ and points $x,x_0\in L$.  For the ordinary Hilbert kernel,
\begin{align*}
 &\left|\int f_n(y)
 \left(\frac1{x-y}-\frac1{x_0-y}\right)dy\right|\\
 &\qquad\le |x-x_0|
 \int_{|y|>R_n}\frac{dy}{|x-y|\,|x_0-y|}
 \le \frac{C_L}{R_n}.
\end{align*}
For the nonlinear part, let
\[
 M_L=\max_{x\in L}k(x),
 \qquad
 m_L=\min_{x\in L}k(x).
\]
On the right tail,
\begin{align*}
 &\int_{R_n}^{\infty}k'(y)
 \left|\frac1{k(x)-k(y)}-\frac1{k(x_0)-k(y)}\right|dy\\
 &\qquad\le |k(x)-k(x_0)|
 \int_{R_n}^{\infty}\frac{k'(y)}{(k(y)-M_L)^2}\,dy\\
 &\qquad=\frac{|k(x)-k(x_0)|}{k(R_n)-M_L}.
\end{align*}
The corresponding left-tail estimate is
\[
 \frac{|k(x)-k(x_0)|}{m_L-k(-R_n)}.
\]
Since $k$ is an increasing homeomorphism of $\R$ onto itself, both expressions tend to zero.  Consequently,
\begin{equation}
 \sup_{x,x_0\in L}
 |\Kop_kf_n(x)-\Kop_kf_n(x_0)|\longrightarrow0. \label{eq:tail-local-constant}
\end{equation}

Compactness gives a subsequence converging in $\BMO$ to some $[u]$.  If $A_1,A_2$ are any two bounded intervals, choose one compact interval $L$ containing them.  Equation \eqref{eq:tail-local-constant} and Lemma~\ref{lem:average-functional} imply $u_{A_1}=u_{A_2}$.  Hence $u$ is constant a.e. and $[u]=0$, contradicting the uniform lower bound $\varepsilon$.
\end{proof}

Let $k\in\SQS(\R)$ and set
\[
 q=\log k'\in\BMO(\R).
\]
Fix a quasisymmetry control function $\eta$ for $k$, so that
\begin{equation}
 \frac{|k(x)-k(a)|}{|k(x)-k(b)|}
 \le
 \eta\left(\frac{|x-a|}{|x-b|}\right) \label{eq:qs-control}
\end{equation}
whenever the points are distinct.  In particular, $\eta(t)\to0$ as $t\downarrow0$.

For a bounded interval $I$ define
\[
 \alpha_k(I)=
 \sup\left\{
 \|\Kop_kf\|_{\BMO}:
 \|f\|_\infty\le1,
 \ \supp f\subset I
 \right\}.
\]

The next proposition is the central estimate. 
\begin{proposition}\label{prop:detector}
For every $\delta>0$ there are constants
\[
 c_\delta>0,
 \qquad
 \kappa_\delta>0,
\]
depending only on $\delta$, $\|q\|_{\BMO}$, and the quasisymmetry function $\eta$, with the following property.  If $I$ is a bounded interval,
\[
 M(q,I)\ge\delta,
\]
and $G\subset I$ is measurable with
\[
 |G|\le\kappa_\delta|I|,
\]
then there exists a real-valued function $f$ such that
\[
 |f|\le\chi_{I\setminus G}
\]
and
\[
 \|\Kop_kf\|_{\BMO}\ge c_\delta.
\]
In particular,
\[
 M(q,I)\ge\delta
 \quad\Longrightarrow\quad
 \alpha_k(I)\ge c_\delta.
\]
\end{proposition}

\begin{proof}
Write $I=[a,b]$ and $L=|I|$.  Put
\[
 u=q-q_I.
\]
Since $u_I=0$,
\[
 \int_Iu_+=\int_Iu_-
 =\frac12\int_I|u|
 \ge\frac{\delta L}{2}. \label{eq:positive-negative-mass}
\]
Define
\[
 E=\left\{y\in I:q(y)>q_I+\frac\delta5\right\},
 \qquad
 F=\left\{y\in I:q(y)<q_I-\frac\delta5\right\}.
\]
Then,
\[
 \int_Eu_+\ge\frac{3\delta L}{10},
 \qquad
 \int_Fu_-\ge\frac{3\delta L}{10}.
\]
The John--Nirenberg inequality implies
\[
 \frac1L\int_I|u|^2\,dy
 \le C\|q\|_{\BMO}^2.
\]
Cauchy--Schwarz therefore gives a number $\sigma_\delta>0$, depending only on $\delta$ and $\|q\|_{\BMO}$, such that
\begin{equation}
 |E|\ge\sigma_\delta L,
 \qquad
 |F|\ge\sigma_\delta L. \label{eq:large-high-low}
\end{equation}
Choose
\[
 0<\kappa_\delta<\frac{\sigma_\delta}{4}.
\]
Then
\[
 E'=E\setminus G,
 \qquad
 F'=F\setminus G
\]
satisfy
\begin{equation}
 |E'|,|F'|\ge\frac{3\sigma_\delta}{4}L. \label{eq:remaining-size}
\end{equation}

Split $E'$ into a left half $E_1$ and a right half $E_2$ of equal measure, ordered so that every point of $E_1$ lies to the left of every point of $E_2$, up to null sets.  Split $F'$ similarly as $F_1\cup F_2$.  Either
\begin{equation}
 E_1<F_2 \label{eq:order-one}
\end{equation}
or
\begin{equation}
 F_1<E_2. \label{eq:order-two}
\end{equation}
Assume \eqref{eq:order-one}; the other case is the right-hand analogue.

Choose an integer $N>2$, to be fixed momentarily, and let
\[
 J=I-NL=[a-NL,b-NL].
\]
For $x\in J$ and $y\in I$ define
\begin{equation}
 R_x(y)=\frac{k'(y)(y-x)}{k(y)-k(x)}>0. \label{eq:def-Rxy}
\end{equation}
If $y_E\in E_1$ and $y_F\in F_2$, then $y_E<y_F$ and
\[
 \frac{k'(y_E)}{k'(y_F)}
 \ge e^{2\delta/5}.
\]
Furthermore,
\[
 \frac{y_E-x}{y_F-x}\ge\frac{N-1}{N+1},
 \qquad
 \frac{k(y_F)-k(x)}{k(y_E)-k(x)}\ge1.
\]
Therefore
\begin{equation}
 \frac{R_x(y_E)}{R_x(y_F)}
 \ge e^{2\delta/5}\frac{N-1}{N+1}. \label{eq:R-ratio}
\end{equation}
Choose $N=N(\delta)$ so large that the right-hand side of \eqref{eq:R-ratio} is a number $\lambda>1$.  Set
\[
 \tau=\frac{\lambda-1}{\lambda+1}>0.
\]
Then, for every $x\in J$, at least one of the following alternatives holds:
\begin{align}
 R_x(y)&\ge1+\tau
 &&\text{for a.e. }y\in E_1, \label{eq:E-alternative}\\
 R_x(y)&\le1-\tau
 &&\text{for a.e. }y\in F_2. \label{eq:F-alternative}
\end{align}
Indeed, if \eqref{eq:E-alternative} fails, then
\[
 \essinf_{E_1}R_x<1+\tau,
\]
and \eqref{eq:R-ratio} gives
\[
 \esssup_{F_2}R_x
 <\frac{1+\tau}{\lambda}=1-\tau.
\]

Let
\[
 f_E=\chi_{E_1},
 \qquad
 f_F=\chi_{F_2}.
\]
Since $x<y$ on $J\times I$, Lemma~\ref{lem:kernel} and \eqref{eq:def-Rxy} give
\begin{equation}
 \Kop_kf(x)=\frac1\pi\int_I
 \frac{1-R_x(y)}{y-x}f(y)\,dy. \label{eq:left-kernel-R}
\end{equation}
If \eqref{eq:E-alternative} holds, then
\[
 \Kop_kf_E(x)
 \le-\frac\tau\pi\int_{E_1}\frac{dy}{y-x}
 \le-\frac{\tau|E_1|}{\pi(N+1)L}.
\]
If \eqref{eq:F-alternative} holds, then
\[
 \Kop_kf_F(x)
 \ge\frac{\tau|F_2|}{\pi(N+1)L}.
\]
By \eqref{eq:remaining-size}, both $|E_1|$ and $|F_2|$ are bounded below by a fixed positive multiple of $L$.  Hence there is $A_\delta>0$ such that, for every $x\in J$, either
\[
 \Kop_kf_E(x)\le-A_\delta
\]
or
\[
 \Kop_kf_F(x)\ge A_\delta.
\]
It follows that one of the two functions $f_E,f_F$, call it $f$, has the following property: there is a measurable set $J_0\subset J$ with
\begin{equation}
 |J_0|\ge\frac L2 \label{eq:J0-size}
\end{equation}
and either
\begin{equation}
 \Kop_kf\ge A_\delta\quad\text{on }J_0, \label{eq:positive-output}
\end{equation}
or
\begin{equation}
 \Kop_kf\le-A_\delta\quad\text{on }J_0. \label{eq:negative-output}
\end{equation}

This is next compared with the value on a farther interval.  Let $M>N+2$ and set
\[
 J'=I-ML.
\]
For $x\in J'$, $|f|\le\chi_I$ and Lemma~\ref{lem:kernel} imply
\begin{align}
 |\Kop_kf(x)|
 &\le\frac1\pi\int_I\frac{k'(y)}{k(y)-k(x)}\,dy
 +\frac1\pi\int_I\frac{dy}{y-x} \notag\\
 &\le\frac1\pi\frac{k(b)-k(a)}{k(a)-k(x)}
 +\frac1{\pi(M-1)}. \label{eq:far-output-1}
\end{align}
Since $x\le a-(M-1)L$, quasisymmetry gives
\begin{equation}
 \frac{k(b)-k(a)}{k(a)-k(x)}
 \le
 \frac{k(b)-k(a)}{k(a)-k(a-(M-1)L)}
 \le
 \eta\left(\frac1{M-1}\right). \label{eq:far-output-2}
\end{equation}
Choose $M=M(\delta,k)$ so large that the right-hand side of \eqref{eq:far-output-1}, after using \eqref{eq:far-output-2}, is at most $A_\delta/4$.  Thus
\begin{equation}
 |\Kop_kf|\le\frac{A_\delta}{4}
 \quad\text{on }J'. \label{eq:small-on-Jprime}
\end{equation}

Let $Q$ be the smallest interval containing $J\cup J'$.  Then $|Q|\le(M+1)L$.  Suppose \eqref{eq:positive-output} holds.  For every constant $c$, if $c\le A_\delta/2$, then
\[
 |\Kop_kf-c|\ge A_\delta/2
 \quad\text{on }J_0;
\]
if $c>A_\delta/2$, then \eqref{eq:small-on-Jprime} gives
\[
 |\Kop_kf-c|\ge A_\delta/4
 \quad\text{on }J'.
\]
Using \eqref{eq:J0-size},
\[
 \frac1{|Q|}\int_Q|\Kop_kf-c|\,dx
 \ge\frac{A_\delta}{8(M+1)}.
\]
Taking $c=(\Kop_kf)_Q$ gives a fixed lower bound for the BMO norm.  The case \eqref{eq:negative-output} is identical.

It remains to write out the case \eqref{eq:order-two}.  Thus
$F_1<E_2$.  Put
\[
 J=I+NL=[a+NL,b+NL]
\]
and, for $x\in J$ and $y\in I$, define
\begin{equation}
 \widetilde R_x(y)
 =\frac{k'(y)(x-y)}{k(x)-k(y)}>0.
 \label{eq:def-Rxy-right}
\end{equation}
If $y_F\in F_1$ and $y_E\in E_2$, then $y_F<y_E$ and
\[
 \frac{k'(y_E)}{k'(y_F)}\ge e^{2\delta/5}.
\]
Moreover,
\[
 \frac{x-y_E}{x-y_F}\ge\frac{N-1}{N+1},
 \qquad
 \frac{k(x)-k(y_F)}{k(x)-k(y_E)}\ge1.
\]
Consequently
\[
 \frac{\widetilde R_x(y_E)}{\widetilde R_x(y_F)}
 \ge e^{2\delta/5}\frac{N-1}{N+1}=\lambda.
\]
With the same $\tau=(\lambda-1)/(\lambda+1)$, for each $x\in J$ either
\[
 \widetilde R_x(y)\ge1+\tau
 \quad\text{for a.e. }y\in E_2,
\]
or
\[
 \widetilde R_x(y)\le1-\tau
 \quad\text{for a.e. }y\in F_1.
\]
For $x>y$, the kernel formula becomes
\begin{equation}
 \Kop_kf(x)=\frac1\pi\int_I
 \frac{\widetilde R_x(y)-1}{x-y}f(y)\,dy.
 \label{eq:right-kernel-R}
\end{equation}
Thus $f_E=\chi_{E_2}$ produces a positive lower bound whenever the first
alternative holds, while $f_F=\chi_{F_1}$ produces a negative lower bound
whenever the second alternative holds.  As before, one of these two functions,
call it $f$, has a fixed-sign lower bound of magnitude $A_\delta$ on a subset
$J_0\subset J$ of measure at least $L/2$.

For the far comparison interval take
\[
 J'=I+ML.
\]
If $x\in J'$, then
\begin{align*}
 |\Kop_kf(x)|
 &\le\frac1\pi\int_I\frac{k'(y)}{k(x)-k(y)}\,dy
 +\frac1\pi\int_I\frac{dy}{x-y}\\
 &\le\frac1\pi\frac{k(b)-k(a)}{k(x)-k(b)}
 +\frac1{\pi(M-1)}.
\end{align*}
Since $x\ge b+(M-1)L$, quasisymmetry gives
\[
 \frac{k(b)-k(a)}{k(x)-k(b)}
 \le
 \frac{k(b)-k(a)}{k(b+(M-1)L)-k(b)}
 \le\eta\left(\frac1{M-1}\right).
\]
The same choice of a sufficiently large $M$ therefore makes
$|\Kop_kf|\le A_\delta/4$ on $J'$.  Taking the smallest interval containing
$J\cup J'$ and repeating the two-case comparison with an arbitrary constant
$c$ gives the same BMO lower bound as before.

In both orderings, $f$ is a characteristic function of a subset of $E'$ or
$F'$.  Hence $|f|\le\chi_{I\setminus G}$, and the proof is complete.
\end{proof}

The converse implication in Theorem~\ref{thm:main} is proved by establishing all three conditions \eqref{eq:CMO-small}--\eqref{eq:CMO-tail}.

Let $h\in\SQS(\T)$ and assume $P_h^-$ is compact.  Set $k=\gamma^{-1}\circ h\circ\gamma$, and put
\[
 q=\log k'.
\]
By Proposition~\ref{prop:circle-line-transfer}, $P_{k,\R}^-$ is compact.  Proposition~\ref{prop:Ph-H-relation} then implies that $\Kop_k$ is compact on $\BMO(\R)$.

\subsection{The global small-interval condition}

\begin{proposition}\label{prop:small-condition}
One has
\[
 \lim_{r\downarrow0}\sup_{|I|<r}M(q,I)=0.
\]
\end{proposition}

\begin{proof}
Suppose not.  Then there are $\delta>0$ and intervals $I_n$ such that
\[
 |I_n|\to0,
 \qquad
 M(q,I_n)\ge\delta.
\]
By Proposition~\ref{prop:detector}, there are $f_n$ with
\[
 |f_n|\le\chi_{I_n},
 \qquad
 \|\Kop_kf_n\|_{\BMO}\ge c_\delta.
\]
If the centers $c(I_n)$ remain bounded, pass to a subsequence for which $c(I_n)\to x_0$.  Lemma~\ref{lem:shrinking} gives
\[
 \|\Kop_kf_n\|_{\BMO}\to0,
\]
a contradiction.  If the centers are unbounded, pass to a subsequence with $|c(I_n)|\to\infty$.  Since $|I_n|\to0$, the supports of $f_n$ escape every compact subset of $\R$, and Lemma~\ref{lem:tail} gives the same contradiction.
\end{proof}

\subsection{Fixed-scale intervals escaping to infinity}

\begin{proposition}\label{prop:translation-condition}
For every fixed bounded interval $I_0$,
\[
 \lim_{|x|\to\infty}M(q,I_0+x)=0.
\]
\end{proposition}

\begin{proof}
Otherwise, for some $\delta>0$ there are $|x_n|\to\infty$ such that
\[
 M(q,I_0+x_n)\ge\delta.
\]
Proposition~\ref{prop:detector} gives $f_n$ supported in $I_0+x_n$ with
\[
 \|\Kop_kf_n\|_{\BMO}\ge c_\delta.
\]
The supports escape to infinity, contradicting Lemma~\ref{lem:tail}.
\end{proof}

\subsection{The large-interval condition}

\begin{proposition}\label{prop:large-condition}
One has
\[
 \lim_{R\to\infty}\sup_{|I|>R}M(q,I)=0.
\]
\end{proposition}

\begin{proof}
Suppose the conclusion fails.  Then there are $\delta>0$ and intervals $I_n$ such that
\[
 L_n=|I_n|\to\infty,
 \qquad
 M(q,I_n)\ge\delta.
\]
Choose numbers $R_n\to\infty$ with
\[
 \frac{R_n}{L_n}\to0;
\]
for example, $R_n=\sqrt{L_n}$.  Set
\[
 G_n=I_n\cap[-R_n,R_n].
\]
Then
\[
 |G_n|\le2R_n=o(L_n).
\]
For all sufficiently large $n$,
\[
 |G_n|\le\kappa_\delta|I_n|,
\]
where $\kappa_\delta$ is the constant in Proposition~\ref{prop:detector}.  The deletion-stable conclusion of that proposition gives $f_n$ such that
\[
 |f_n|\le\chi_{I_n\setminus G_n}
\]
and
\[
 \|\Kop_kf_n\|_{\BMO}\ge c_\delta.
\]
But
\[
 \supp f_n\subset\R\setminus[-R_n,R_n].
\]
Lemma~\ref{lem:tail} therefore implies
\[
 \|\Kop_kf_n\|_{\BMO}\to0,
\]
a contradiction.
\end{proof}

Propositions~\ref{prop:small-condition}, \ref{prop:translation-condition}, and \ref{prop:large-condition}, together with the CMO characterization recalled in Section 2, give
\begin{equation}
 \log k'=q\in\CMO(\R). \label{eq:q-CMO}
\end{equation}
Thus $k$ belongs to the strongly vanishing symmetric class $\SSS_0(\R)$ of \cite{LiuShen}.  By \eqref{eq:SS-circle-line}, the homeomorphism $h$ belongs to $\SSS(\T)$.   This proves the converse implication in Theorem~\ref{thm:main}.

\begin{corollary}[Real-line form]\label{cor:real-line}
Let $k\in\SQS(\R)$.  Then
\[
 P_{k,\R}^-:\BMOA(\Hh)\longrightarrow\BMOA(\Hh)
\]
is compact if and only if
\[
 \log k'\in\CMO(\R),
\]
equivalently, if and only if $k\in\SSS_0(\R)$.
\end{corollary}

\end{document}